\documentclass[11pt]{article}
\usepackage{amsmath}
\usepackage{amssymb}
\usepackage{amsthm}
\usepackage{bbm}
\usepackage{graphicx,psfrag}
\usepackage{enumerate}
\usepackage{natbib}
\usepackage{url} 
\usepackage{bm}
\usepackage[bottom]{footmisc}
\usepackage{booktabs}
\usepackage{caption}
\usepackage{hyperref}
\usepackage{color}
\usepackage{mathtools}
\usepackage{geometry}
\newtheorem{corollary}{Corollary}

\newtheorem{theorem}{Theorem}
\newtheorem{definition}{Definition}
\theoremstyle{definition}
\newtheorem{remark}{Remark}
\newtheorem{problem}{Problem}
 
\title{A hypothesis-testing perspective on the G-normal distribution theory}
\author{Shige Peng\footnote{Institute of Mathematics, Shandong University, Jinan, Shandong 250100, China.} ~and~Quan Zhou\footnote{Department of Statistics, Texas A\&M University, College Station, Texas 77843, U.S.A.}}
\date{}
\DeclareMathOperator*{\Var}{Var}

\newcommand{\pr}{\mathsf{P}}

\newcommand{\E}{\mathsf{E}} 

\newcommand{\bbR}{\mathbb{R}} 
 
\newcommand{\smax}{\overline{\sigma}}
\newcommand{\smin}{\underline{\sigma}}

\newcommand{\cF}{\mathcal{F}} 
\newcommand{\cH}{\mathcal{H}} 
\newcommand{\cP}{\mathcal{P}} 

\newcommand{\bbE}{\mathbb{E}}
\newcommand{\define}{\coloneqq} 
\newcommand{\ind}[1]{\mathbbm{1}(#1)}

\begin{document}

\maketitle

\begin{abstract}
The G-normal distribution was introduced by~\citet{peng2007g} as the limiting distribution in the central limit theorem for sublinear expectation spaces. Equivalently, it can be interpreted as the solution to a stochastic control problem where we have a sequence of random variables, whose variances can be chosen based on all past information. 
In this note we study the tail behavior of the G-normal distribution through analyzing a nonlinear heat equation. 
Asymptotic results are provided so that the tail ``probabilities" can be easily evaluated with high accuracy. 
This study also has a significant impact on the hypothesis testing  theory for heteroscedastic data; we show that even if the data are generated under the null hypothesis, it is possible to cheat and attain statistical significance by sequentially manipulating the error variances of the observations. 
\end{abstract}


%\vfill
%\spacingset{1.1}

%\newpage 

\section{Introduction}\label{sec:intro}
The primary goal of this note is to study the asymptotic tail behavior of the G-normal distribution, providing a key result to the theory of sublinear expectation spaces developed by~\cite{peng2008new}. To statisticians, our result can be interpreted from a hypothesis-testing perspective.  Suppose for heteroscedastic observations $X_1, X_2, \dots$, one wants to conduct a statistical test regarding their common mean. Then by manipulating their variances, the experimenter is able to reject the null hypothesis with probability greater than the nominal significance level when the data are actually generated under the null.  
This can be seen as a new type of ``cheating with the data", which in spirit is similar to the well-known ``p-hacking" phenomenon\footnote{The term ``p-hacking" refers to the the phenomenon that researchers may try out different data analysis methods until they obtain a p-value small enough.}~\citep{head2015extent}.

As suggested by its name, G-normal distribution plays a central role in the sublinear expectation theory as normal distribution does in the classical probability theory. Indeed, it is the limiting distribution in the generalized ``central limit theorem" for sublinear expectation spaces. A more detailed review of the G-normal distribution (and sublinear expectation spaces) will be given in Section~\ref{sec:sublinear}.   As noted in~\cite{fang2017limit}, to characterize the tail behavior of the G-normal distribution, equivalently we can consider the following stochastic control problem (see also Theorem~\ref{th:clt} and Definition~\ref{def:gnormal}.)

\begin{problem}\label{prob:control}
Let $\epsilon_1, \epsilon_2, \dots$ be a sequence of i.i.d. random variables such that $\E(\epsilon_i) = 0$, $\E(\epsilon_i^2) = 1$ and $\E(|\epsilon_i|^3) < \infty$, defined on some filtered probability space $(\Omega, \cF, \pr, \{\cF_i\}_{i=0}^\infty)$ where $\{\cF_i\}_{i=0}^\infty$ is the natural filtration generated by $\{\epsilon_i\}_{i=1}^\infty$, i.e. $\cF_i = \sigma( \epsilon_1, \dots, \epsilon_i)$. 
Let $\Sigma(\smin, \smax)$ be the collection of all predictable sequences with respect to $\{\cF_i\}_{i=0}^\infty$ that always take value in $[\smin, \smax]$ where $\smin, \smax$ are given constants ($0 \leq \smin \leq \smax < \infty$.)   
For any $\{\sigma_i\}_{i=1}^n \in \Sigma(\smin, \smax) $, define $X_i = \sigma_i \epsilon_i$ and $\bar{X}_n = (X_1 + \cdots + X_n) / n$. The problem is to compute the following two functions and find the sequences $\{\sigma_i\}_{i=1}^\infty$ that attain the corresponding supremums, 
\begin{equation}\label{eq:tailc}
\begin{aligned}
p_1(c; \smin, \smax) \define &\, \lim\limits_{n \rightarrow \infty} \sup\limits_{ \{\sigma_i\} \in \Sigma(\smin, \smax) }  \E [ \ind{ \sqrt{n} \bar{X}_n > c }   ],  \\
p_2(c; \smin, \smax) \define &\, \lim\limits_{n \rightarrow \infty} \sup\limits_{ \{\sigma_i\} \in \Sigma(\smin, \smax) }  \E  [ \ind{  \sqrt{n} |\bar{X}_n| > c }   ], 
\end{aligned}
\end{equation}
where $c \in [0, \infty)$ and $\mathbbm{1}$ denotes the indicator function.  
\end{problem}

If $\smin = \smax = \sigma$, the observations $X_1, X_2, \dots$ are i.i.d. and thus by the classical central limit theorem, we have $p_2(c) = 2\Phi(-c/\sigma) = 2p_1(c)$ where $\Phi$ denotes the distribution function of the standard normal distribution. When $\smin < \smax$, the functions $p_1$ and $p_2$ are called tail capacities of the G-normal distribution, where ``capacity" can be understood as a generalization of probability. The characterization of $p_1$ and $p_2$ is vital to the understanding of G-normal distribution. 
To evaluate $p_1$ and $p_2$, we need solve a nonlinear heat equation, which is studied in Section~\ref{sec:pval}. It turns out that $p_1$ admits a closed-form expression but $p_2$ does not. The main technical result of this paper is an asymptotic approximation for $p_2$, which is highly accurate and very easy to compute. 
 
Now we explain how Problem~\ref{prob:control} relates to hypothesis testing.  Suppose we observe $X_1, X_2, \dots, $, which are generated by the model given in Problem~\ref{prob:control} and consider the null hypothesis $H_0: \E(X_i) = 0$ for every $i$. When $\smin$ is slightly smaller than $\smax$, both the heteroscedasticity (i.e. the fact that $\Var(X_i)$ is not a constant) and the dependence structure of the observations could be very difficult to detect; if $\{X_i\}_{i=1}^n$ is treated as an i.i.d. sample, the null hypothesis can be tested using the t-statistic, 
\begin{equation}\label{eq:def.T}
T_n(X)  = \dfrac{  \sqrt{n} \bar{X}_n }{ \sqrt{s^2_n} } , \quad s^2_n = \dfrac{1}{n - 1} \sum\limits_{i=1}^n (X_i - \bar{X}_n)^2. 
\end{equation}
%which under the null should follow a t-distribution with $n - 1$ degrees of freedom.  
For sufficiently large $n$, $T_n(X)$ can be treated as a standard normal variable and the probability of $s^2_n$ being greater than $\smax^2$ quickly decreases to zero.  Hence, for a one-sided test with level $\alpha$, the null hypothesis would be rejected if $ \sqrt{n} \bar{X}_n > \smax  \Phi^{-1}(1 - \alpha)$. (One can also use the $(1-\alpha)$\% quantile of the t-distribution here and our theory will apply equally.) 
Imagine that an experimenter is able to choose any $\{\sigma_i \}_{i=1}^n$ from the set $\Sigma(\smin, \smax)$ (defined in Problem~\ref{prob:control}) and wants to maximize the probability of the event $\{ \sqrt{n} \bar{X}_n > \smax  \Phi^{-1}(1 - \alpha) \}$.
% can be maximized by solving Problem~\ref{prob:control} and computing $p_1( \smax  \Phi^{-1}(1 - \alpha) )$. 
Then, as will be shown in Section~\ref{sec:application}, the asymptotically optimal strategy is to simply choose either $\sigma_i = \smin$ or $\sigma_i = \smax$ depending on whether $X_1 + \cdots + X_{i - 1}$ is greater than $\sqrt{n}  \smax  \Phi^{-1}(1 - \alpha).$ Further, $p_1( \smax  \Phi^{-1}(1 - \alpha) )$ is always strictly greater than $\alpha$ given that $\smax > \smin$. A similar analysis can be conducted for the two-sided test as well. Simulation studies  with unknown $\smin, \smax$  will be provided in Section~\ref{sec:application}. 

%Note that it is the dependence structure rather than the heteroscedasticity of the observations that causes the Z-test (or the t-test)   fails to control the type I error rate under $\alpha$. If $\{X_i\}_{i=1}^\infty$ is heteroscedastic but independent, the Z-test is always asymptotically valid by  Lyapunov's central limit theorem.  

We point out that in many applications, it is possible for the experimenter to affect the error variances. For example,  consider an economist planning to survey individuals of different ages to study whether some variable has an effect on personal income. 
 The errors are heteroscedastic because the income of older people tends to have a larger variance. Whether the economist deliberately surveys more younger (or older) people seems unimportant since  age is included in the regression model as a confounding variable.  But the result of this paper implies that this is not true  if the economist decides who to survey next (in terms of age) based on previous observations.

\section{G-normal distribution and Peng's central limit theorem}\label{sec:sublinear} 
The sublinear expectation theory was motivated by capturing the model uncertainty in real-world markets~\citep{artzner1999coherent, chen2002ambiguity} and has found applications in economics, mathematical finance and statistics~\citep{epstein2014ambiguous, yang2018improving, lin2016k}. Concepts such as ``distribution" and ``independence" are redefined for a sublinear expectation space. 
But to make the present note easier to understand, we will present all the results using the language of classical probability theory, except the use of the terms ``G-normal distribution" and ``tail capacity". 

The central limit theorem for sublinear expectation spaces,  first developed by~\cite{peng2008new}, has been formulated in various ways. 
In Theorem~\ref{th:clt} we present the version given in~\citet{rokhlin2015central1}~\citep[see also][]{fang2017limit},  which can be seen as a generalization of the classical central limit theorem to controlled stochastic processes. It is an immediate corollary of Peng's original central limit theorem, but translated into the language of classical probability (see Appendix B).

\begin{theorem}\label{th:clt}
Let $\{\epsilon_i\}_{i=1}^\infty$, $\{\sigma_i\}_{i=1}^n$, $\{X_i\}_{i=1}^n$ and $\Sigma(\smin, \smax)$ be as given in Problem~\ref{prob:control}. 
Then for any Lipschitz function $\varphi$, 
\begin{equation}\label{eq:clt}
\lim\limits_{n \rightarrow \infty} \sup_{ \{\sigma_i\} \in \Sigma(\smin, \smax) }    \E \Big[ \varphi (  \sqrt{n} \bar{X}_n  )  \Big] = u(1, 0; \varphi), 
\end{equation}
where $\{u(t, x; \, \varphi)\colon \,  (t, x) \in [0, \infty) \times \bbR \}$ is the unique viscosity solution to the Cauchy problem, 
\begin{equation}\label{eq:def.G-heat}
u_t = \dfrac{1}{2} \left(  \smax^2 (u_{xx})^+ -  \smin^2 (u_{xx})^- \right),  \quad \quad  u(0, x) =  \varphi(x). 
\end{equation}
In the above expression, $u_t = \partial u/\partial t$, $u_{xx} = \partial^2 u / \partial x^2$, and the superscripts $+$ and $-$ denote the positive and negative parts respectively. 
\end{theorem}

\begin{remark}\label{remark:conds}
See~\citet{peng2008new},~\cite{peng2019law},~\citet{rokhlin2015central1} or \citet[Theorem 4.1]{fang2017limit}  for the proof. %  and~\citet{denis2011function} for similar results. 
The theorem can be further generalized to non-identically distributed sequence $\{\epsilon_i\}_{i=1}^\infty$. For the convergence rate of Theorem~\ref{th:clt}, see~\citet{fang2017limit}, \citet{song2019normal}\footnote{The paper of~\citet{song2019normal} was submitted in 2017 and the earliest among the four.},~\citet{krylov2019shige} and~\citet{huang2019monotone}. 
\end{remark}

\begin{remark}\label{remark:indicator}
The function $\varphi$  can also be any Borel-measurable indicator function~\citep{peng2010nonlinear, yang2018improving}. 
\end{remark}

\begin{remark}\label{remark:Gheat}
The classical central limit theorem corresponds to the special case $\smin = \smax = \sigma$, where~\eqref{eq:def.G-heat} reduces to the heat equation and thus  
\begin{align*}
u(t, x; \varphi) = \int_{-\infty}^\infty  \phi(z)  \varphi(x +  \sigma \sqrt{t}  z ) dz   
\end{align*}
where  $\phi(x)$ is the  density function of the standard normal distribution.  We call the nonlinear partial differential equation~\eqref{eq:def.G-heat} a G-heat equation.
\end{remark}

Now we are ready to define G-normal distribution. Note that G-normal distribution is not a distribution in the traditional sense, and a ``random variable" $Z$ following G-normal distribution actually has distributional uncertainty. 
 
\begin{definition}[G-normal distribution]\label{def:gnormal}
Let $\cP_Z$ be a set of  probability measures defined on the space $(\Omega, \cF)$.
A measurable function $Z\colon  \Omega \mapsto \bbR$  is said to follow a G-normal distribution with lower variance $\smin^2$ and upper variance $\smax^2$ $(0 \leq \smin \leq \smax)$,   if, for every Lipschitz function $\varphi$, 
\begin{equation*}\label{eq:def.G-normal}
%\cE [ \varphi(Z)] \define  
\sup\limits_{\pr  \in \cP_Z} \E_{\pr} [ \varphi(Z) ] =  \sup\limits_{\pr  \in \cP_Z} \int_\Omega \varphi(Z) \, d \pr  = u(1, 0; \varphi),
\end{equation*}
where $u(1, 0; \varphi)$ is as given in Theorem~\ref{th:clt}. 
\end{definition}

\begin{remark}\label{remark:gnormal}
As expected, when $\smin = \smax = \sigma$, the G-normal distribution reduces to the normal distribution $N(0, \sigma^2)$. 
\end{remark}

\begin{remark}\label{remark:sublinear}
In the sublinear expectation theory, $\cP_Z$ represents the collection of all possible probability measures underlying $Z$,  and  $u(1, 0; \varphi)$ is called the sublinear expectation of $\varphi(Z)$.  
For the existence of G-normal distribution, see~\citet[Theorem 2.1 and Chapter II.2]{peng2010nonlinear}. 
\end{remark}
  
When  $\varphi$ in Definition~\ref{def:gnormal} is an indicator function, we have 
\begin{align*}
p(A; \smin, \smax) \define \sup_{P \in \cP_Z} \int_\Omega  \ind{Z \in A} d \pr 
= \sup_{P \in \cP_Z} \pr \{  Z \in A \}. 
\end{align*}
Here $p(A) = p(A; \smin, \smax)$ is a set function defined for each $A \in \cF$, which clearly satisfies $p(\emptyset) = 0$ and $p(\Omega) = 1$.   
But unlike a probability measure, $p(\cdot)$ is not an additive function; it is called a Choquet capacity~\citep{choquet1954theory}, or capacity for short~\citep[see also][]{chen2005choquet, denis2011function}.
In this work, we are interested in the ``tail capacities" of the G-normal distribution, i.e. the functions $p_1$ and $p_2$ defined in~\eqref{eq:tailc}. They are the solutions to the G-heat equations~\eqref{eq:def.G-heat} with initial condition $\varphi(x) = \ind{x > c}$ and $\varphi(x) = \ind{|x| > c}$ respectively. 
Understanding the tail behavior of the G-normal distribution is crucial to the asymptotic theory of sublinear expectation spaces. 

For a formal exposition of the sublinear expectation theory, readers are referred to~\citet{peng2010nonlinear}. For the purposes of this paper, we only need to use Theorem~\ref{th:clt} to compute the tail capacities of the G-normal distribution.% As we explained in Section~\ref{sec:intro}, this would also help us understand how false positive results may arise from ``cheating" on the error variances of a sequential experiment. 

\section{Tail capacities of the G-normal distribution}\label{sec:pval}
By Theorem~\ref{th:clt} and Definition~\ref{def:gnormal}, to compute the tail capacities of the G-normal distribution, we need to solve the corresponding Cauchy problems given in~\eqref{eq:def.G-heat}, which is often difficult since the G-heat equation is  nonlinear. 
Fortunately,  we have a closed-form solution for the one-sided tail capacity, $p_1(c)$. 
But for the two-sided tail capacity $p_2(c)$, a  closed-form solution is not available and we will offer an asymptotic approximation which has remarkable accuracy for large values of $c$.  
Recall that  $\phi$ and $\Phi$ respectively  denote the probability density function and the cumulative distribution function of  the standard normal distribution. 

\subsection{One-sided tail capacity}
We first   derive a closed-form solution to the G-heat equation~\eqref{eq:def.G-heat}  with initial condition $u (0, x)  = \ind{x > c}$. This result will be used later for approximating the solution to the G-heat equation with initial condition  $u (0, x) = \ind{|x| > c}$.

\begin{theorem}\label{th:one}
Consider the following Cauchy problem for the G-heat equation, 
$$u_t = \dfrac{1}{2} \left(  \smax^2 (u_{xx})^+ -  \smin^2 (u_{xx})^- \right),  \quad \quad  u(0, x) =  \ind{x > c},$$
where $c \in \bbR$.  
The solution is given by $u(t, x) = f \{ (x - c) / \sqrt{t} \}$ where 
\begin{equation*}\label{eq:sol.u}
f(y) = f(y; \smin, \smax)= \dfrac{2}{ \smax + \smin }  \int_{-y}^\infty \big \{   \phi  (z/\smax ) \ind{z \geq 0}  \, + \,  \phi  (z/\smin ) \ind{z < 0}  \big \} dz. 
\end{equation*}
\end{theorem}

\begin{proof}
This result is also given in~\citet{yang2018improving} but without derivation. 
Here we show how the solution is derived. 
For the uniqueness and the existence of this solution, see~\citet{peng2010nonlinear, yang2018improving}. 

Notice that if $u(t, x)$ is the solution to this Cauchy problem, then, for any $a > 0$, $u(at, \sqrt{a}(x - c) + c)$ is also a solution.
This implies that $u(t, x) = f \{ (x - c) / \sqrt{t} \}$ for some function $f$. 
Then routine calculations show that the function $f = f(y)$ must satisfy 
\begin{equation*}
 - y  f_y = \smax^2 (f_{yy})^+  -  \smin^2 (f_{yy})^-. 
\end{equation*}
If $u$ is twice differentiable in $x$, $f(y)$ must be continuous at $0$ and thus it can be expressed as 
\begin{align*}
f(y) =   \int_{-y}^\infty \Big \{ \dfrac{A_1}{\smax}  \phi  (z/\smax ) \ind{z \geq 0}  \, + \, \dfrac{A_2}{\smin} \phi  (z/\smin ) \ind{z < 0}  \Big \} dz + B, 
\end{align*}
where $A_1, A_2, B_1, B_2$ are some constants to be determined. 
By the initial conditions, $f(\infty) = 1$ and $f( -\infty) = 0$, we obtain that $A_1 + A_2 = 2$ and $B = 0$. 
Finally, the twice differentiability of $f(y)$  amounts to matching the left and right first derivatives at $y = 0$,  which yields 
$A_1/A_2 = \smax/\smin$. 
The theorem is then proved by checking that $u(t, x) = f \{ (x - c) / \sqrt{t} \}$ is in $C^{1,2}$. 
\end{proof}

\begin{corollary}\label{cor:p1}
The one-sided tail capacity of the G-normal distribution is
\begin{align*}
 p_1(c; \smin, \smax) = \dfrac{2}{ \smax + \smin }  \int_c^\infty \big \{   \phi  (z/\smax ) \ind{z \geq 0} 
 \, + \,   \phi  (z/\smin ) \ind{z < 0}  \big \} dz. 
\end{align*}
\end{corollary}
\begin{proof}
This is immediate from the definition of G-normal distribution. 
\end{proof}

\begin{remark}
One can check that  $p_1(c)$ behaves just like a tail probability function in the sense that, for any $0 < \smin \leq \smax$, $p_1(c)$ is monotone decreasing with $p_1(-\infty) = 1$ and $p_1(\infty) = 0.$
\end{remark}

\begin{corollary}\label{cor:v} 
The solution to the following Cauchy problem,
 $$v_t = \dfrac{1}{2} \left(  \smax^2 (v_{xx})^+ -  \smin^2 (v_{xx})^- \right),  \quad \quad  v(0, x) =  \ind{x < - c},$$
 is given by 
$v(t, x) = f\{ (- x   - c) / \sqrt{t} \}$ where $f$ is as given in Theorem~\ref{th:one}. 
\end{corollary}

\begin{proof}
This follows immediately from Theorem~\ref{th:one} by symmetry. 
\end{proof}

\subsection{Two-sided tail capacity}
To compute the two-sided tail capacity $p_2(c)$ defined in~\eqref{eq:tailc}, we need to solve the G-heat equation~\eqref{eq:def.G-heat} with initial condition $u(0, x) = \ind{|x| > c}$,  which does not admit a closed-form solution. However, we do have an asymptotic result that turns out to be very useful. 

\begin{theorem}\label{th:two}
Let $w(t, x)$ be the solution to the Cauchy problem, 
$$w_t = \dfrac{1}{2} \left(  \smax^2 (w_{xx})^+ -  \smin^2 (w_{xx})^- \right),  \quad \quad  w(0, x) =  \ind{|x| > c}.$$
Let $u(t, x) = f\{ (x - c) / \sqrt{t} \}$  be as given in Theorem~\ref{th:one} and $v(t, x) = f\{ (- x - c) / \sqrt{t} \}$ be as given in Corollary~\ref{cor:v}. 
For any $x \in \bbR$ and $t \geq 0$, if  $c >   \smax \sqrt{t}   / 2$,    
\begin{align*} 
& 0 \leq  u(t, x) + v(t, x) - w(t, x) \leq \dfrac{2 (\smax - \smin) }{\smax} \Phi \left( - \dfrac{ 2c  }{  \smax \sqrt{t} } \right). 
 %  \quad \bar{\delta}(t) =  \dfrac{ 2 c (\smax - \smin)   }{ \smax^2 \sqrt{  t^3} } \phi \left(  \dfrac{ 2c   }{ \smax \sqrt{t} } \right). 
\end{align*}
\end{theorem}

\begin{proof}
The goal is to show that for any fixed $(t, x) \in [0, 1] \times \bbR$ and sufficiently large $c$, we  have $w(t, x) \approx u(t, x) + v(t, x)$ (it may be helpful to think of the heat transfer model to gain some intuition.)
To this end, define 
\begin{align*}
G(x) = \dfrac{1}{2} ( \smax^2  x^+  - \smin^2 x^- ) ,\quad  \delta(t, x) = G(u_{xx}) + G(v_{xx}) - G(u_{xx} + v_{xx}).
\end{align*}
Observe that $G$ is a sublinear function and thus $\delta(t, x) \geq 0$. Further, $u + v$ satisfies the initial condition that $u(0, x) + v(0, x) = \ind{|x| > c}.$
Now if we can find a function $\bar{\delta} (t)$ such that $\delta(t, x) \leq \bar{\delta}(t)$ for every $x$, we can apply the comparison theorem given in~\citet[Appendix C]{peng2010nonlinear}   to get 
\begin{equation}\label{eq:comp}
0  \leq u(t, x) + v(t, x)  - w(t, x)  \leq  \int_0^t \bar{\delta}(s) ds,
\end{equation}
for all $t \geq 0, \, x \in \bbR$. 
It is straightforward to check that 
\begin{equation}\label{eq:delta}
\delta(t, x) = \left\{ \begin{array}{cc}
\dfrac{\smax^2 - \smin^2}{2} \min\{ |u_{xx}|, |v_{xx}| \} ,  &  \text{ if  } \; u_{xx} v_{xx} < 0,  \\
0  ,  &  \text{ if  } \; u_{xx} v_{xx} \geq 0.
\end{array}
\right. 
\end{equation}
Hence we only need to bound $\min\{ |u_{xx}|, |v_{xx}| \}.$
Direct calculations give that 
\begin{align*}
u_{xx}(t, x) \;&= \dfrac{1}{t } f_{yy} \left(  \dfrac{ x   -c}{\sqrt{t} } \right) , \quad \quad 
v_{xx}(t, x) = \dfrac{1}{t } f_{yy} \left(  \dfrac{ - x   -c}{\sqrt{t} } \right) , \\ 
f_{yy} (y) \;&= 
 \dfrac{ - 2y }{ \smax + \smin }    \left \{  \dfrac{ 1}{\smax^2}  \phi  (y/\smax ) \ind{y \leq 0}  \, + \,  \dfrac{1}{\smin^2} \phi  (y/\smin ) \ind{y > 0}  \right \}  . 
\end{align*}
Notice that we only need to bound $\min\{ |u_{xx}|, |v_{xx}| \}$ on the region where $u_{xx} v_{xx} < 0$, which is   $(c , \infty) \cup (-\infty, - c)$. 
For  $x > c $, we have 
\begin{equation}\label{eq:ineq}
\min\{ |u_{xx}|, |v_{xx}| \} \leq |v_{xx}| \leq \sup\limits_{x > c } v_{xx}(t, x) =
\dfrac{1}{t } \sup\limits_{x > c }  f_{yy} \left(  \dfrac{ - x  -c}{\sqrt{t} } \right) . 
\end{equation}
Assume $c > \smax \sqrt{t} / 2$.  Then for any $x > c $, $ (- x  -c) / \sqrt{t} < - \smax $.
Observing that $\sup_{y<0} |f_{yy}(y)|$ is attained at $y = - \smax$ and $|f_{yy}(y)|$ is monotone increasing on $(-\infty, -\smax)$, 
we can bound the right-hand side of~\eqref{eq:ineq} by 
\begin{equation}\label{eq:ineq2}
\min\{ |u_{xx}|, |v_{xx}| \} \leq \dfrac{1}{t }   f_{yy} \left(  \dfrac{  -2c   }{\sqrt{t} } \right). 
\end{equation}
The same argument shows that this bound also holds for $x < -c $ and thus holds for $x \in \bbR$. 
%Finally, notice that if $c > \smax t  /2$ for some $t  > 0$, then $ c > \smax s / 2$   for any $s \in [0, t]$. 
%The theorem then can be proved using~\eqref{eq:comp}, \eqref{eq:delta} and~\eqref{eq:ineq2}. 
Using \eqref{eq:delta} and~\eqref{eq:ineq2} and the assumption $c > \smax \sqrt{t} / 2$, we obtain the bound 
\begin{equation*}
\bar{\delta}(s) = \dfrac{ \smax^2 - \smin^2 }{2s }  f_{yy} \left(  \dfrac{  -2c   }{\sqrt{s} } \right)  = \dfrac{ 2 c (\smax - \smin)   }{ \smax^2 \sqrt{ s^3} }  \phi \left(  \dfrac{ 2c   }{ \smax \sqrt{s} } \right)
\end{equation*}
for any $s \in [0, t]$. Integrating $\bar{\delta}(s)$ with respect to $s$ gives 
\begin{equation}\label{eq:err}
\int_0^t  \bar{\delta}(s) ds    = \dfrac{2 (\smax - \smin) }{\smax} \Phi \left( - \dfrac{ 2c  }{  \smax \sqrt{t} } \right). 
%\leq  \dfrac{ 2M_c  \sqrt{t} }{  \lambda_c^2 \sqrt{2\pi} }   e^{- \lambda_c^2 / 2t}.
\end{equation}
The theorem is then proved by recalling~\eqref{eq:comp}. 
\end{proof}
 
Theorem~\ref{th:two} suggests that we may approximate the two-sided tail capacity using the sum of two one-sided tail capacities. For any fixed $t$, as $c \rightarrow \infty$,  the error term~\eqref{eq:err} quickly goes to zero due to the fast decay of the function $\Phi$.  However, for any $(t, x)$,  $u(t, x) + v(t, x)$ also goes to zero  as $c \rightarrow \infty$.
%Hence, whether $u(t, x) + v(t, x)$ provides a good approximation to $w(t, x)$ is still not clear.  
The next corollary confirms that the relative error of this approximation is negligible.

\begin{corollary}\label{cor:w2}
Under the setup of Theorem~\ref{th:two} and assuming  $c >  \smax    / 2$,  we have 
\begin{equation*}\label{eq:asymp}
0 \leq  u(t, x) + v(t, x) - w(t, x) \leq  
 \dfrac{(\smax - \smin) \sqrt{t} }{c \sqrt{2 \pi}  } \exp\left( - \frac{ 2c^2 }{  \smax^2 t} \right) . 
\end{equation*}
Further, the relative error is given by 
\begin{equation*}
0 \leq  \dfrac{ u(t, x) + v(t, x) - w(t, x) }{u(t, x) + v(t, x)} \leq  
 \dfrac{(\smax^2 - \smin^2)(c^2/\smax^2 + t)}{4c^2}   \exp\left( - \frac{ 3c^2 }{ 2 \smax^2 t } \right)   . 
\end{equation*}
\end{corollary}
\begin{proof}
It is well known that the tail probabilities of the standard normal distribution can be bounded by  
\begin{align*}
\frac{x}{x^2 + 1} \leq \frac{  \Phi( - x) }{ \phi(x)} \leq \frac{1}{x}. 
\end{align*}
Hence the error term~\eqref{eq:err} can be bounded by 
\begin{align*}
\int_0^t  \bar{\delta}(s) ds   \leq  \dfrac{ (\smax - \smin) \sqrt{t}}{ c \sqrt{2\pi} } \exp\left( - \frac{ 2c^2 }{  \smax^2 t} \right) . 
\end{align*}
By straightforward calculations, one can show  that $u(t, x) + v(t, x)$ is minimized at $x = 0$.  By Theorem~\ref{th:one}, for $c \geq 0$, 
\begin{align*} 
u(t, 0) = v(t, 0) =  \dfrac{2\smax}{\smax + \smin}    \Phi \left( - \frac{ c }{  \smax \sqrt{t} }  \right) 
\geq   \dfrac{\sqrt{2} }{\sqrt{\pi} (\smax + \smin)} \dfrac{c \smax^2 \sqrt{t} }{c^2 + \smax^2 t}   \exp\left( - \frac{  c^2 }{ 2 \smax^2 t} \right), 
\end{align*}
which then yields the relative error bound stated in the theorem. 
\end{proof}

\begin{corollary}\label{cor:two}
Let $p_1, p_2$ be the tail capacities of the G-normal distribution as given in~\eqref{eq:tailc}.  As $c \rightarrow \infty$, we have $p_2(c) \sim 2p_1(c)$ with the relative error given by 
\begin{align*}
\dfrac{ | p_2(c) - 2p_1(c)  | }{ 2p_1(c) }  \lesssim \dfrac{1 - \smin^2/\smax^2}{4} \exp\left( -\frac{3c^2}{2\smax^2} \right), 
\end{align*}
where $\sim$ denotes asymptotic equivalence and $\lesssim$ means ``asymptotically less than". 
\end{corollary}

\begin{proof}
Notice that, by Definition~\ref{def:gnormal}, $2 p_1(c) = u(1, 0) + v(1, 0)$ and $p_2(c) = w(1, 0)$. 
Hence we only need to apply Corollary~\ref{cor:w2} with $t = 1$. 
\end{proof}

\begin{remark}
This approximation is accurate enough for usual purposes. 
For example,  using the bound given in Corollary~\ref{cor:w2}, for $\smax  = 1, \smin = 0.8$, we have $p_2( \Phi^{-1}(0.95) ) = 0.11$ with RE (relative error) $< 0.002$, $p_2( \Phi^{-1}(0.975) ) = 0.056$ with 
RE $< 4 \times 10^{-4}$, and $p_2( \Phi^{-1}(0.995) ) = 0.011$ with RE $< 5 \times 10^{-6}$.  
\end{remark}

\section{Hypothesis testing with heteroscedastic data}\label{sec:application}
To completely solve  Problem~\ref{prob:control}, we need to find the asymptotically optimal policies $\{\sigma_i\}_{i=1}^\infty$ that attain the supremums in~\eqref{eq:tailc}. 
This can be most conveniently computed using the stochastic control theory and Hamilton-Jacobi-Bellman   equation, which we briefly explain in Appendix A. 

Now we come back to the hypothesis testing problem described in Section~\ref{sec:intro}. First, consider the one-sided test. 
Since all the observations $X_1, X_2, \dots$ have mean zero and variance less than or equal to $\smax^2$,  as $n\rightarrow \infty$, $s_n^2$ is less than or equal to $\smax^2$ with probability one. Hence, if we conduct a test with rejection region $\{T_n(X) > \Phi^{-1}(1-\alpha) \}$,  the probability of rejecting the null hypothesis, $\E(X_i) = 0$, can be  bounded from below by $\pr ( \sqrt{n} \bar{X}_n  >  \smax \Phi^{-1} (1-\alpha)  )  $.  
However, if the experimenter is able to choose any $\{\sigma_i\} \in \Sigma(\smin, \smax)$, then 
%which is asymptotically valid if $X_1, X_2, \dots$ are independent, 
\begin{equation}\label{eq:test}
  \lim\limits_{n \rightarrow \infty} \sup\limits_{ \{\sigma_i\} \in \Sigma(\smin, \smax) }   \pr ( \sqrt{n} \bar{X}_n  >  \smax \Phi^{-1} (1-\alpha)  )  
=   p_1 (   \smax \Phi^{-1} (1-\alpha); \smin, \smax ),  %= \dfrac{2 \alpha}{1 + \smin/ \smax}, 
\end{equation}
and by Theorem~\ref{th:one},  for any $\alpha \leq 0.5$ and $\smin < \smax$, 
\begin{equation}\label{eq:p1.alpha}
p_1 (   \smax \Phi^{-1} (1-\alpha); \smin, \smax ) = \dfrac{2 \alpha}{1 + \smin/ \smax} > \alpha. 
\end{equation}
Note that, if $X_1, X_2, \dots$ are independent, the test is still asymptotically valid in the sense that the type I error rate is $\alpha$ as $n \rightarrow \infty$. By violating the independence assumption and carefully manipulating the sequence $\{\sigma_i\}_{i=1}^\infty$, the experimenter is able to increase the probability of falsely rejecting the null hypothesis by at least $2 \smax / (\smin + \smax)$. 
By~\eqref{eq:opt} and Theorem~\ref{th:one}, one can show that the asymptotically optimal strategy to attain the supremum in~\eqref{eq:test} is 
\begin{equation}\label{eq:opt1} 
\begin{aligned}
\sigma_i  =   \smax,   & \quad \quad \text{ if }   S_{i-1} / \sqrt{n} \leq \smax \Phi^{-1} (1 - \alpha),  \\
\sigma_i  =   \smin,   & \quad \quad \text{ if }   S_{i-1} / \sqrt{n} >  \smax \Phi^{-1}(1 - \alpha),  
\end{aligned}
\end{equation} 
where $S_i = X_1 + \cdots + X_i$. 

The two-sided test can be analyzed similarly. The probability of rejecting the null hypothesis  can be  bounded from below by $\pr ( \sqrt{n}  |\bar{X}_n | >  \smax \Phi^{-1} (1-\alpha/2)  )$. By Corollary~\ref{cor:two}, if $\smin < \smax$ and $\alpha$ is small, 
\begin{align*}
p_2 (   \smax \Phi^{-1} (1-\alpha/2); \smin, \smax ) \approx 
2 p_1 ( \smax \Phi^{-1} (1-\alpha/2); \smin, \smax  ) = \dfrac{2 \alpha}{1 + \smin/ \smax} > \alpha. 
\end{align*}
The  control policy given in~\eqref{eq:opt} can still be expressed in the form like~\eqref{eq:opt1}; however, $S_i$ is replaced by $|S_i|$ and the threshold changes with $i$. 
Numerically we can compute the threshold for each $i$ by solving $w_{xx} = 0$ where $w$ is as in Theorem~\ref{th:two}.  We observe that,  for small $\alpha$,  the threshold goes to $\Phi^{-1}(1 - \alpha / 2)$ very quickly as $i$ increases and thus can be treated just as a constant. 

Nevertheless, for a finite sample size, we need to take into account the distribution of the sample variance $s^2_n$, and the Z-test should be replaced by the t-test.  More importantly, in practice %even if the experimenter knows how to manipulate the variances of observations, 
the parameters $\smin$ and $\smax$ are usually unknown.  Inspired by our theoretical results, we propose the following heuristic control policy for a two-sided test, 
\begin{equation}\label{eq:opt2} 
\begin{aligned}
\sigma_i  =   \smax,   & \quad \quad \text{ if }   |S_{i-1}| /  \sqrt{ n s_{i-1}^2}  \leq   c_\alpha  ,  \\
\sigma_i  =   \smin,   & \quad \quad \text{ if }   |S_{i-1}| /   \sqrt{ n s_{i-1}^2}  >   c_\alpha   . 
\end{aligned}
\end{equation} 
where $c_\alpha$ is the critical value  and $s_i^2$ is the sample variance of $X_1, \dots, X_i$.

We end our paper with a realistic simulation study.  
Consider a two-sided t-test with $\alpha = 0.05$. %which yields $p_2(  \Phi^{-1} (1 - \alpha/2) ) \approx 0.0556.$ 
We generate normal samples $X_1, \dots, X_n$ (with mean zero and variances $\sigma_1^2, \dots, \sigma_n^2$) using~\eqref{eq:opt2} with $\smax = 1$ and $\smin = 0.8$. 
Then we compute the t-statistic by~\eqref{eq:def.T} (assuming $\smin, \smax$ are unknown) and perform a two-sided t-test with degree of freedom equal to $n - 1$.  We repeat this experiment for $1$ million times. For $n = 20$, the null hypothesis $\E(X_i) = 0$ is rejected in $5.65$ percent of all the experiments; for $n = 200$, the null is rejected in $5.89$ percent of all the experiments.  
To further illustrate how the manipulation of the variances of $X_1, \dots, X_n$ affects the sampling distribution of the t-statistic, we compare the empirical distribution of $T_n(X)$ with t-distribution in Figure~\ref{fig:tcdf} for $n = 200$. It can be seen that the empirical distribution of $T_n(X)$ almost coincides with the theoretical t-distribution, except near the critical values (approximately $\pm 2$). 

In the simulation we use  $\smax = 1$ and $\smin = 0.8$ to reflect that in reality the possible influence from the experimenter is limited. Because $\smin, \smax$ are close, every simulated set of observations $\{X_i\}_{i=1}^n$   looks just like a homoscedastic normal sample.  Further, without prior knowledge, the dependence structure is almost impossible to detect. But compared to the nominal significance level $\alpha = 0.05$,  the type I error rate is inflated by $13$\% for $n=20$ and $18$\% for $n = 200$.

\begin{figure}
\begin{center}
\includegraphics[width=0.9\linewidth]{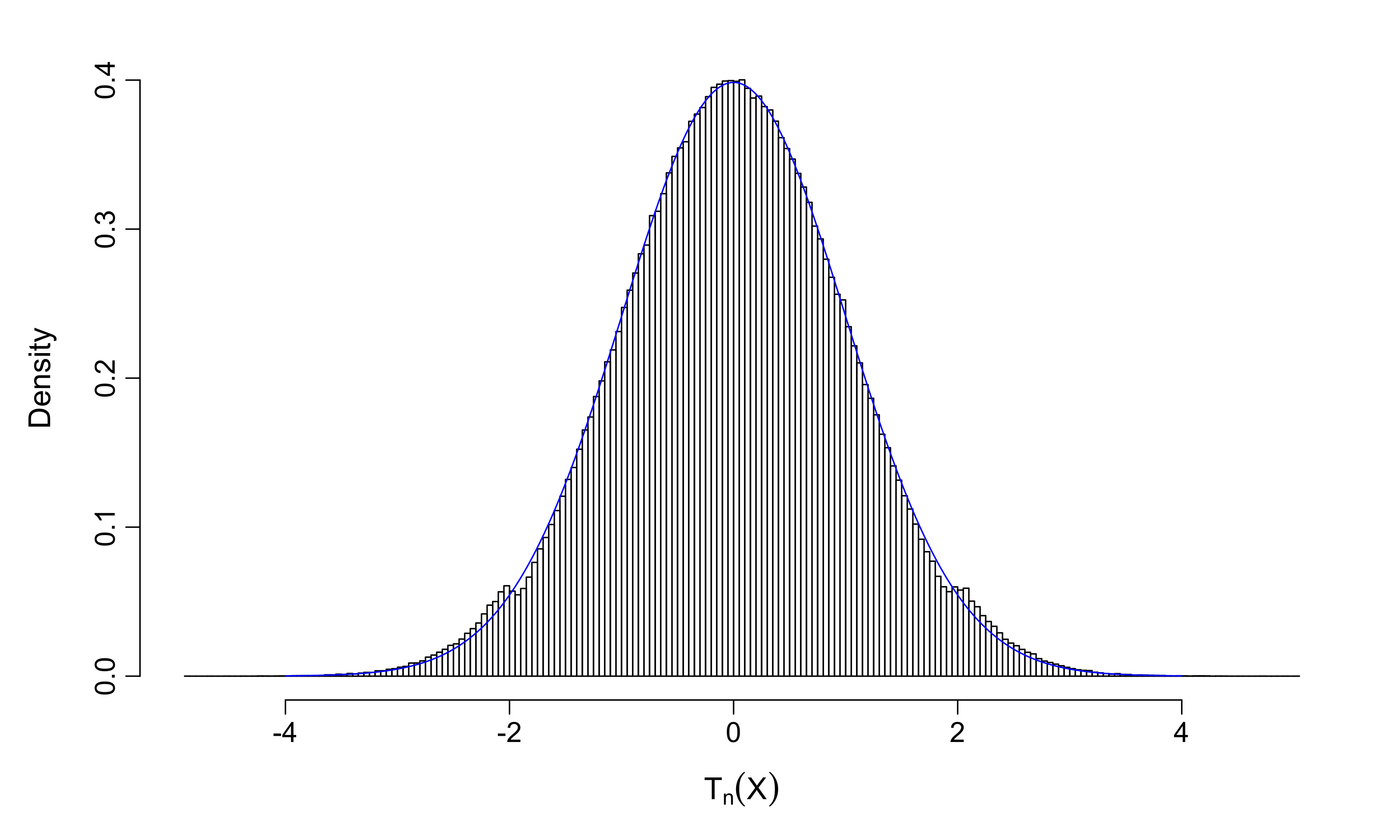}
\caption{ Histogram of the empirical distribution of the t-statistic $T_n(X)$ for $n=200$.  The normal observations $X_1, \dots, X_n$ are generated sequentially using~\eqref{eq:opt2}. The blue curve represents the density function of the t-distribution with $n-1$ degrees of freedom. For $n=200$ and $\alpha = 0.05$, the critical values are $\pm 1.97$. }\label{fig:tcdf}
\end{center}
\end{figure}

\setcounter{equation}{0}
\renewcommand{\theequation}{A\arabic{equation}}
\section*{Appendix A}
The asymptotically optimal policy $\{\sigma_i\}_{i=1}^\infty$ that attains the supremums in~\eqref{eq:tailc} can be derived using the diffusion limit of Problem~\ref{prob:control}. %, which can be formulated as follows. 
Let $W_t$ be a Wiener process and $Y_t$ be a controlled process evolving by  $d Y_t =   \sigma_t  d W_t,$
where $\sigma_t  $ is a process progressively measurable  with respect to the natural filtration generated by $W_t$. %$\{\cF_t \}_{t \geq 0}$.  
Let $S_i = X_1 + \cdots + X_i$ where $\{X_i\}_{i=1}^\infty$ are as defined in Problem~\ref{prob:control}. Intuitively speaking, as $n \rightarrow \infty$, $S_i / \sqrt{n}$ has the same distribution as $Y_{i/n}$, and in particular,  $\sqrt{n} \bar{X}_n$ has the same distribution as $Y_1$ (recall Donsker's Theorem.)

Consider the value function, 
\begin{equation*}\label{eq:val}
V  (t, y; \varphi) = \sup_{\sigma_t \in  [\smin, \smax] } \E [  \varphi(Y_1)  \mid Y_{t} = y ],  \quad \quad (t , y) \in [0, 1] \times \bbR, 
\end{equation*}
where the supremum is taken over all the progressively measurable processes  that take value in $[\smin, \smax]$ on the time interval $  [0, 1]$.  In our case,  $\varphi$ is given by $\varphi(y) = \ind{y > c}$ or $\varphi(y) = \ind{|y| > c}$. 
It is not difficult to see that the optimal control $\sigma_t$ must be a measurable function of $(t, Y_t)$ by a Markovian argument.  Then we may guess the solution by solving the Hamilton-Jacobi-Bellman  (HJB) equation, 
\begin{equation}\label{eq:hjb}
\sup\limits_{\sigma_t \in [\smin, \smax]}  \left\{  \dfrac{\partial V}{\partial t} + \dfrac{\sigma^2_t}{2} \dfrac{\partial^2 V}{\partial y^2} \right\} = 0,  \quad \quad V(0,  y) = \varphi(y), 
\end{equation}
and then prove it using the so-called verification techniques~\citep[for more details see, for example,][]{yong1999stochastic}. As expected, the HJB approach yields the same result as Theorem~\ref{th:clt}, and indeed, the solution to~\eqref{eq:hjb} is given by 
\begin{equation}\label{eq:vu}
V(t, y; \varphi) = u (1 - t, y; \varphi), 
\end{equation}
where $u$ is the solution to the G-heat equation~\eqref{eq:def.G-heat}. 

The optimal control process $\sigma_t$ is the one which attains the supremum in~\eqref{eq:hjb}; that is, 
\begin{equation*}
\begin{aligned}
\sigma_t =   \smax,   & \quad \quad \text{ if }   V_{yy} (t, Y_t) \geq 0,  \\
\sigma_t =   \smin,   & \quad \quad \text{ if }   V_{yy} (t, Y_t) < 0. 
\end{aligned}
\end{equation*} 
Using~\eqref{eq:vu}, the corresponding discrete-time approximation is given by 
\begin{equation}\label{eq:opt}
\begin{aligned}
\sigma_i =   \smax,   & \quad \quad \text{ if }   u_{xx} ( 1-(i-1)/n, \, S_{i-1}/\sqrt{n} ) \geq 0,  \\
\sigma_i =   \smin,   & \quad \quad \text{ if }   u_{xx} ( 1-(i-1)/n, \, S_{i-1}/\sqrt{n} ) < 0,  
\end{aligned}
\end{equation} 
for $i = 1, \dots, n$. 

\setcounter{equation}{0}
\renewcommand{\theequation}{B\arabic{equation}}
\section*{Appendix B}

Here we provide a brief presentation of Peng's central limit theorem~\citep{peng2008new, peng2019law} and show that it immediately implies Theorem~\ref{th:clt}. 
We refer the readers to~\cite{peng2010nonlinear} for further details. 
Let $\cH$ be a linear space of real-valued functions defined on a set $\Omega$ such that if $X_1, \dots,  X_n \in \cH$, we have $\varphi(X_1, \dots, X_n) \in \cH$ for any Lipschitz function $\varphi$. Given a collection of probability measures $\cP$,   we can define a sublinear expectation (of some random variable $X$), denoted by $\bbE$, by 
\begin{equation*} 
\bbE [ X ] \define   \sup\limits_{\pr  \in \cP } \E_{\pr} [ X ] =  \sup\limits_{\pr  \in \cP } \int_\Omega X  \, d \pr. 
\end{equation*}
One can check that, for $X, Y \in \cH$ with $\bbE\, |X| , \bbE\, |Y| < \infty$,  we have (i) $\bbE[X] \geq \bbE [Y]$ if $X \geq Y$; (ii) 
$\bbE[X] - \bbE[ - Y] \leq \bbE[X + Y] \leq \bbE[X] + \bbE[Y]$; (iii) $\bbE[ \lambda X] = \lambda \bbE[X]$ for $\lambda \geq 0$;
(iv) $\bbE[X+ c] = \bbE[X] + c$ for $c \in \bbR$. 

In the sublinear expectation theory, two random variables, $X, Y \in \cH$, are called identically distributed (under $\bbE$) iff $\bbE[\varphi(X) ] = \bbE[ \varphi(Y) ]$ for every Lipschitz function $\varphi$. For  $X_1, \dots, X_n \in \cH$, $X_n$ is said to be independent from $Y = (X_1, \dots, X_{n - 1})$ (under $\bbE$) iff  $\bbE[\varphi(X_n, Y)] = \bbE[  \bbE[  \varphi(X_n, y) ]_{y = Y}  ]$ for every Lipschitz $\varphi$. 
We say $X_1, X_2, \dots, $ are i.i.d. if they are identically distributed and $X_n$ is independent from $X_1, \dots, X_{n-1}$ for each $n$. 
 
\begin{theorem}[Peng's central limit theorem]
Let $X_1, X_2, \dots$ be an i.i.d. sequence of random variables under sublinear expectation $\bbE$. 
If $\bbE\, |X_1|^3  < \infty$, $\bbE[X_1] = \bbE[ - X_1] = 0$, $\bbE[ X_1^2 ] = \smax^2$ and $- \bbE[ - X_1^2] = \smin^2$ for some $0 < \smin \leq \smax < \infty$, 
then $\{X_i\}_{i=1}^\infty$ converges ``in distribution" to the G-normal distribution; that is, for any Lipschitz function $\varphi$, 
\begin{align*}
\lim\limits_{n \rightarrow \infty}  \bbE \left[  \varphi(  \sqrt{n} \, \bar{X}_n  )  \right] = \bbE[ \varphi(Z) ], 
\end{align*}
where $Z$ is G-normally distributed with lower variance $\smin^2$ and upper variance $\smax^2$. 
\end{theorem}  

As mentioned in  Remark~\ref{remark:sublinear}, $\bbE[\varphi(Z)]$ can be computed by solving the corresponding G-heat equation. 
To see that the above result immediately implies Theorem~\ref{th:clt}, we only need to find an appropriate sublinear expectation space and check the conditions. 

To this end, let $\Omega \define \left\{ \omega=(\omega_{i})_{i=1}^{\infty},\omega_{i}\in\mathbb{R}\right\} $, and 
$$\cH \define  \left\{ \xi(\omega)=\varphi(\omega_{1},\cdots,\omega_{n}),\;\omega\in\Omega, \,  n\in\mathbb{N},\varphi\in C_{Lip}(\mathbb{R}^{n})\right\}.$$ 
Define $X_i(\omega) \define \omega_i$ for $i = 1, 2, \dots$, which we call the canonical processes of $(\Omega, \cH)$. Clearly $X_i(\omega) \in \cH$.  
For any random variable $\xi(\omega) = \varphi(X_1, \dots, X_n)$ where $\varphi \in C_{Lip}(\mathbb{R}^{n})$ and for any $\psi \in C_{Lip}(\mathbb{R})$, we define a sublinear expectation $\bbE$ by 
\begin{equation}\label{eq:def.sub.peng}
\bbE [ \psi(\xi(\omega) ) ] = \sup_{ \{\sigma_i\} \in \Sigma(\smin, \smax) } \E_{\pr} [  \psi( \varphi(\sigma_1 \epsilon_1, \dots, \sigma_n \epsilon_n ) )], 
\end{equation}
where $\epsilon_1, \dots, \epsilon_n$ and $\Sigma(\smin, \smax)$ are as given in Problem~\ref{prob:control} (and $X_i = \sigma_i \epsilon_i$.) It is clear that 
$\bbE[ \psi(X_1)] = \bbE[ \psi(X_i)] = \sup_{\sigma_i \in [\smin, \smax]} \E  [ \psi(\sigma_i \epsilon_i )]$ for any Lipschitz $\psi$. 
Further,  one can show that $X_n$ is independent of $X_1, \dots, X_{n-1}$ under $\bbE$ using the definition of $\Sigma(\smin, \smax)$. Finally, all the moment conditions are satisfied by the properties of classical normal distribution. 
Applying Peng's central limit theorem with~\eqref{eq:def.sub.peng}, we obtain Theorem~\ref{th:clt}.

\section*{Acknowledgements}
We thank the anonymous reviewers for their helpful comments. 

%\section*{References}
{\footnotesize
\setlength{\bibsep}{0.2pt}
\bibliographystyle{plainnat}
\bibliography{ref}
 }

\end{document}